\documentclass[]{amsart}
\usepackage{amsfonts,latexsym,rawfonts,amsmath,amssymb, amsthm,bm}
\usepackage[T1]{fontenc}
\usepackage{latexsym,lscape,rawfonts}
\usepackage{tikz-cd}
\usepackage{CJK}
\usepackage{mathrsfs, enumitem}
\usepackage[title]{appendix}
\usepackage{galois}
\usepackage[all,cmtip]{xy}
\usepackage{extarrows,color}
\usepackage{times}

\newtheorem{thm}{Theorem}[section]

\newtheorem*{fixed point criterion}{\fixed point criterion}
\newtheorem{cor}[thm]{Corollary}
\newtheorem{lem}[thm]{Lemma}
\newtheorem{prop}[thm]{Proposition}

\theoremstyle{definition}

\theoremstyle{remark}
\newtheorem{rem}[thm]{Remark}
\numberwithin{equation}{section}

\newcommand{\Z}{\mathbb Z}
\newcommand{\Q}{\mathbb Q}

\newcommand{\fix}{\mathrm{Fix}}

\newcommand{\rk}{\mathrm{rk}}

\newcommand{\aut}{\mathrm{Aut}}

\newcommand{\edo}{\mathrm{End}}

\newcommand{\im}{\mathrm{Im}}

\begin{document}

\title[Classification of fixed subgroups of endomorphisms]{Classification of fixed subgroups of endomorphisms in free-abelian times surface groups}
\author{Ke Wang}
\address{School of Mathematics and Statistics, Xi'an Jiaotong University, Xi'an 710049, CHINA}
\email{keqiyehuopo@stu.xjtu.edu.cn}

\author{Qiang Zhang}
\address{School of Mathematics and Statistics, Xi'an Jiaotong University, Xi'an 710049, CHINA}
\email{zhangq.math@mail.xjtu.edu.cn}

\author{Dongxiao Zhao}
\address{School of Mathematics and Statistics, Xi'an Jiaotong University, Xi'an 710049, CHINA}
\email{zdxmath@stu.xjtu.edu.cn}

\thanks{The authors are partially supported by NSFC (No. 12471066).}

\subjclass[2010]{20F65, 20F34, 57M07.}

\keywords{Free group, surface group, direct product, fixed subgroup, classification}

\date{\today}
\begin{abstract}
In this paper, we first study the endomorphisms of free-abelian times surface groups and give a characterization of when they are injective and surjective. Then, we see that free-abelian times hyperbolic groups are Hopfian but not co-Hopfian. Moreover, we give a complete classification of fixed subgroups of endomorphisms in free-abelian times surface groups, which extends that of automorphisms \cite{LWZ25}. Finally, we study the endomorphisms of free-abelian times non-elementary torsion-free hyperbolic groups and give an equivalent condition for them to contain, up to isomorphism, finitely many end-fixed subgroups.
\end{abstract}
\maketitle

\section{Introduction}
For a group $G$, let $\rk(G)$ denote the \emph{rank}, i.e., the minimal cardinality of its generating set, and let $\aut(G),\edo(G)$ denote the sets of all automorphisms and endomorphisms respectively. For any $\phi\in \edo(G)$, the \emph{fixed subgroup} of $\phi$ is defined as
$$\fix\phi:=\{x\in G\mid \phi(x)=x\}.$$
Moreover, a subgroup $H$ of $G$ is called \emph{aut-fixed} (resp. \emph{end-fixed}) \emph{up to isomorphism} if $H\cong \fix\phi$ for some $\phi\in\aut(G)$ (resp. $\phi\in\edo(G)$).

For a free group $F_n$ of rank $n$, in 1989, Imrich and Turner \cite{IT89} extended a celebrated result of Bestvina and Handel \cite{BH92} (Scott's Conjecture) to every endomorphism $\phi\in\edo(F_n)$: $\rk(\fix\phi)\leq \rk(F_n).$ In 2011, Jiang, Wang and Zhang \cite{JWZ11} proved the similar finiteness for \emph{surface groups}, i.e., the fundamental groups of compact surfaces.
From then on, there were many studies on fixed subgroups and related topics for various groups, for example, \cite{BM16, HW04, LMZ25, MO12, Neu92, Sh99, WZ23}. In particular, there were a lot of researches on direct products and free products \cite{Car22,DV13,LZ23,LWZ25,RV21}.

In the recent paper \cite{LWZ25}, Lei, Wang and Zhang studied the subgroups in free-abelian times surface groups, and gave a complete classification of the aut-fixed subgroups. In this paper, we consider end-fixed subgroups of $\pi_1(\Sigma_g)\times \Z^m$ and $F_n\times \Z^m$, where $\Sigma_g$ denotes an orientable closed surface of genus $g$, and $\Z^m$ denotes the free-abelian group of rank $m$. The fundamental group $\pi_1(\Sigma_g)$ has a canonical presentation:
$$
\pi_1(\Sigma_g)=\langle a_1,b_1,\ldots ,a_g,b_g\mid [a_1,b_1]\cdots [a_g,b_g]\rangle,
$$
for $[a, b]=aba^{-1}b^{-1}$. Note that $g, k, r, s, t\ldots$ denote natural numbers, and hence they are not $\aleph_0$ (the cardinality of the integer set $\Z$) in this paper, unless stated otherwise.

We have the following main results.

\begin{thm}\label{main thm 1-free}
For a subgroup $H$ of $F_n\times\mathbb{Z}^m (n, m\geq 2)$, the following items are equivalent:

\begin{enumerate}
    \item $H$ is aut-fixed up to isomorphism;
    \item $H$ is end-fixed up to isomorphism;
    \item $H$ is not isomorphic to $F_{\aleph_0}, ~F_t\times \mathbb{Z}^m(t> n)$;
    \item $H$ is isomorphic to one of the following groups: $$F_t \times \Z^m(0 \leq t \leq n), ~F_t \times \Z^s(t \geq 0, 0 \leq s \leq m-1), ~F_{\aleph_0} \times \Z^s(1\leq s\leq m).$$
\end{enumerate}
\end{thm}

\begin{thm}\label{main thm 2-surf}
For a subgroup $H$ of $\pi_1(\Sigma_g)\times\mathbb{Z}^m  (g, m\geq 2)$, the following items are equivalent:
\begin{enumerate}
    \item $H$ is aut-fixed up to isomorphism;
    \item $H$ is end-fixed up to isomorphism;
    \item $H$ is not isomorphic to $F_{\aleph_0},  ~F_t\times \mathbb{Z}^m (t \geq 2g), ~\pi_1(\Sigma_k)\times\mathbb{Z}^m (k > g)$;
    \item $H$ is isomorphic to one of the following groups: $$F_t \times \Z^s(t \geq 0, 0 \leq s \leq m-1), ~F_t \times \Z^m(0\leq t <2g), ~F_{\aleph_0} \times \Z^s(1\leq s\leq m),$$
    $$~\pi_1(\Sigma_{k(g-1)+1})\times \Z^s(k \geq 1, 0\leq s \leq m-1), ~\pi_1(\Sigma_g)\times \Z^m.$$
\end{enumerate}
\end{thm}

Note that the above Item (1), Item (3) and Item (4) have been proven to be equivalent in \cite{LWZ25}, and Item (1) implies Item (2) clearly, we only need to prove that Item (2) implies Item (3).
When $m=1$, both of $\pi_1(\Sigma_g)\times\Z$ and $F_n\times\Z$ contain, up to isomorphism, only finitely many fixed subgroups of automorphisms, see \cite[Theorem 4.3, Theorem 5.3]{LWZ25}. In contrast, they contain infinitely many fixed subgroups of endomorphisms. More precisely, we have the following.

\begin{thm}\label{main thm 3-Z free}
Let $H$ be a subgroup  of $F_n\times\Z(n \geq 2)$. Then

 (1) $H$ is aut-fixed up to isomorphism if and only if $H$ is not isomorphic to
 $$F_{\aleph_0}, ~F_t\times\Z(t>n); \quad F_{2k} (2k> n ), ~F_t(t>2n).$$

 (2) $H$ is end-fixed up to isomorphism if and only if $H$ is not isomorphic to
 $$ F_{\aleph_0}, ~F_t\times\Z(t>n).$$
\end{thm}


\begin{thm}\label{main thm 4-Z surf}
 Let $H$ be a subgroup  of $\pi_1(\Sigma_g)\times\Z(g \geq 2)$. Then

 (1) $H$ is aut-fixed up to isomorphism if and only if $H$ is not isomorphic to
 $$F_{\aleph_0},  ~F_t\times\Z(t \geq 2g), ~\pi_1(\Sigma_k)\times\Z(k > g);$$
 $$F_{2n} (n\geq g ), ~F_t(t \geq 4g-1), ~~\pi_1(\Sigma_k)(2g-1\neq k>g).$$

 (2) $H$ is end-fixed up to isomorphism if and only if $H$ is not isomorphic to
$$ F_{\aleph_0},  ~F_t\times\Z(t \geq 2g), ~\pi_1(\Sigma_k)\times\Z(k > g). $$
\end{thm}

Recall that free-abelian groups $\Z^n$ and free groups $F_n(n\geq 1)$ are Hopfian (i.e. every epimorphism is an automorphism) but not co-Hopfian (i.e. every monomorphism is an automorphism) since they contain proper subgroups isomorphic to themselves; every closed surface group  $\pi_1(\Sigma_g)(g\geq 2)$ is both Hopfian and co-Hopfian \cite{Ba62,DVa97,Fr63,He72,Ho31}. Now, we show that free-abelian times surface groups are Hopfian through the residually finiteness of surface groups, and show they are not co-Hopfian by constructing isomorphic proper subgroups.

\begin{cor}\label{Hopf cor}
Let $G=H\times\mathbb{Z}^m (m\geq 1)$ for $H=F_g$ or $H=\pi_1(\Sigma_g)$ $(g\geq 1)$. Then $G$ is Hopfian but not co-Hopfian.	
\end{cor}

Moreover, this result can be generalized to all (Gromov) hyperbolic groups. However, since the question of whether hyperbolic groups are always residually finite remains open, we cannot directly deduce this conclusion from the residually finiteness. In fact, it can be derived from the Hopfian property of hyperbolic groups \cite{FS23} and a corollary in \cite{Hi69}.

\begin{cor}\label{Hopf cor for hyperbolic case}
    Let $G=H\times\mathbb{Z}^m (m\geq 1)$ for $H$ a hyperbolic group. Then, $G$ is Hopfian but not co-Hopfian.
\end{cor}

For every non-elementary torsion-free hyperbolic group $H$, Lei, Wang and Zhang \cite[Theorem 1.6]{LWZ25} provided an equivalent condition for $H\times \Z^k (k \geq2)$ to contain, up to isomorphism, infinitely many aut-fixed subgroups. We adapt their result concerning the aut-fixed subgroups of $H\times \Z^k (k\geq 2)$ into the following theorem about the end-fixed subgroups of $H\times\Z^k (k\geq 1)$ through a similar argument.

\begin{thm}\label{thm hyper. times Z^m finite end-fixed subgps}
    Let $H$ be a non-elementary torsion-free hyperbolic group, and let $\beta_1(H)$ be its first betti number. Then $H\times\Z^k\ (k\geq 1)$ contains, up to isomorphism, finitely many fixed subgroups of endomorphisms if and only if $\beta_1(H)=0$ and $H$ contains, up to isomorphism, finitely many fixed subgroups of endomorphisms.
\end{thm}

\begin{rem}
    Shor \cite{Sh99} showed every hyperbolic group contains, up to isomorphism, only finitely many aut-fixed subgroups. Lei and Zhang \cite{LZ23} extended this result and showed every hyperbolic group contains, up to isomorphism, only finitely many fixed subgroups of monomorphisms. However, the question of whether the same result holds for endomorphisms remains open.
\end{rem}

The paper is organized as follows. In Section \ref{sect 2}, we introduce some important facts on subgroups of $F_n\times \Z^m$ and $\pi_1(\Sigma_g)\times \Z^m$. In Section \ref{sect 3}, we prove the main results for end-fixed subgroups in $F_n\times \Z^m$. In Section \ref{sect 4}, we first study endomorphisms of $\pi_1(\Sigma_g)\times \Z^m$, and then prove Theorem \ref{main thm 2-surf} and Theorem \ref{main thm 4-Z surf}. Finally, in Section \ref{sect 5}, we obtain some results about the endomorphisms of free-abelian times hyperbolic groups.

\section{Preliminary}\label{sect 2}

For later use, we review some facts on subgroups of $F_g\times \Z^k$ and $\pi_1(\Sigma_g)\times \Z^k$ in this section. See \cite[Section 2]{LWZ25} for some proofs and more details.

\begin{lem}[Schreier's formula]\label{F_2 subgp}
Let $H$ be a subgroup of $F_g$ ($g\geq 2$) with index $[F_g : H]=m$. Then $H\cong F_{m(g-1)+1}.$
Moreover, if $H$ is nontrivial and normal with index infinite, then $H\cong F_{\aleph_0}$.
\end{lem}

\begin{lem}[Griffiths, \cite{Gr63}]\label{subgp of surface}
Let $H$ be a subgroup of $\pi_1(\Sigma_g)$ ($g\geq 2$). Then
\begin{enumerate}
  \item $H\cong F_t$ ($t\geq 0$) or $F_{\aleph_0}$, if the index $[\pi_1(\Sigma_g):H]=\infty$;
  \item $H\cong \pi_1(\Sigma_{m(g-1)+1})$, if the index $[\pi_1(\Sigma_g):H]=m<\infty$;
  \item $H$ has finite index in $\pi_1(\Sigma_g)$, if $H$ is nontrivial, finitely generated and normal in
$\pi_1(\Sigma_g)$;
\item $H\cong F_{\aleph_0}$, if $H$ is nontrivial and normal with index infinite in
$\pi_1(\Sigma_g)$.
\end{enumerate}
\end{lem}

Delgado and Ventura \cite[Corollary 1.7]{DV13} gave a complete characterization of subgroups of $F_g\times \Z^k$ as follows.

\begin{lem}[Delgado-Ventura, \cite{DV13}]\label{DV subgp}
Let $G=F_g\times \Z^k$ ($g\geq 2, k\geq 1$). Then

(1) Every finitely generated subgroup of $G$ has the following form:
$$F_t\times \Z^s, ~~t\geq 0, ~0\leq s\leq k;$$

(2) Every infinitely generated subgroup of $G$ has the following form:
$$F_{\aleph_0}\times \Z^s, ~~ ~0\leq s\leq k.$$
\end{lem}

There is a parallel characterization of subgroups of $\pi_1(\Sigma_g)\times \Z^k$ in \cite[Lemma 2.4]{LWZ25}.

\begin{lem}[Lei-Wang-Zhang, \cite{LWZ25}]\label{surface subgp}
Let $G=\pi_1(\Sigma_g)\times \Z^k$ ($g\geq 2, k\geq 1$). Then

(1) Every finitely generated subgroup of $G$ has one of the following forms:
$$F_t\times \Z^s, ~\quad \pi_1(\Sigma_{m(g-1)+1})\times \Z^s, \quad t\geq 0, ~m\geq 1, ~0\leq s\leq k; $$

(2) Every infinitely generated subgroup of $G$ has the following form:
$$F_{\aleph_0}\times \Z^s, ~~ ~0\leq s\leq k.$$
\end{lem}

\begin{rem}
For a direct product of general absolute groups, the structure of its subgroups is more complicated. See \cite{Bri23, BH07, BHMS09} for some deep results.
\end{rem}


\section{Fixed subgroups in $F_n\times\Z^m$}\label{sect 3}

The aim of this section is to prove Theorem \ref{main thm 1-free} and Theorem \ref{main thm 3-Z free}. Let
$$G=F_n\times\mathbb{Z}^m  = \left\langle x_1,\dots,x_n, t_1,\dots,t_m \mid t_it_j=t_jt_i, ~t_ix_k=x_kt_i \right\rangle .$$
Then every element in $G $ has a normal form:
$$u\textbf{t}^\textbf{a} =u(x_1,\dots,x_n)t^{a_1}_1 \cdots t^{a_m}_m,$$
where $\textbf{a} = (a_1,\dots,a_m)\in\mathbb{Z}^m $ is a row vector, and $u=u(x_1,\dots, x_n) $ is a freely reduced word on the
alphabet $X=\{x_1,\dots,x_n\}$.

Both in this section and the following section, for a word $u\in F_n$, let the bold lowercase letter $\textbf{u}$ denote its abelianization, and $\textbf{u}^T$ the transpose of the row vector $\textbf{u}$. For example, if $u=u(x_1,x_2)=x_1x_2x_1\in F_2$, then $\textbf{u}=(2,1)\in\mathbb{Z}^2$. Moreover, let $\mathcal{M}_{m\times n}(\Z)$ and $\mathcal{M}_m(\Z)$ denote the sets of all integer matrices of sizes $n\times m$ and $m\times m$, respectively.

\subsection{Endomorphisms}

Delgado and Ventura studied endomorphisms of $G$ and gave the following proposition \cite[Proposition 5.1]{DV13}.

\begin{prop}[Delgado-Ventura, \cite{DV13}]\label{DV endo-prop}
Let $G=F_n\times\mathbb{Z}^m $ with $n\ne 1$. The following is a complete list of all endomorphisms of $ G $:
	\begin{enumerate}
		\item $ \Psi_{\phi,\textbf{Q},\textbf{P}}=u\textbf{t}^\textbf{a} \mapsto \phi(u)\textbf{t}^{\textbf{aQ}+\textbf{uP}} $, where $ \phi\in \edo(F_n), \textbf{Q}\in \mathcal{M}_m(\mathbb{Z}),$ and $ \textbf{P}\in \mathcal{M}_{n\times m}(\mathbb{Z})$.
		
		\item$\Psi_{z,\textbf{l,h,Q,P}}=u\textbf{t}^\textbf{a}\mapsto z^{\textbf{al}^T+\textbf{uh}^T}\textbf{t}^{\textbf{aQ}+\textbf{uP}}$, where $1\ne z\in F_n$ is not a proper power, $\textbf{Q}\in \mathcal{M}_m(\mathbb{Z}) $, $ \textbf{P}\in \mathcal{M}_{n\times m}(\mathbb{Z}),\textbf{0}\ne \textbf{l} \in \mathbb{Z}^m$, and $\textbf{h}\in \mathbb{Z}^n$.
	\end{enumerate}
	In both cases, $\textbf{u}\in \mathbb{Z}^n$ denotes the abelianization of the word $ u \in F_n $.	
\end{prop}

\subsection{Fixed subgroups}
Now we have the following.

\begin{thm}\label{Thm 3.1}
Let $\Psi$ be an endomorphism of $F_n\times\mathbb{Z}^m $. Then
$$\fix\Psi \not\cong F_{\aleph_0}, ~ F_g\times\mathbb{Z}^m (g> n).$$
\end{thm}

\begin{proof}
Since Imrich and Turner \cite{IT89} extended the celebrated result of Bestvina and Handel \cite{BH92} to $\rk(\fix\phi)\leq \rk(F_n)$ for every endomorphism $\phi$ of $F_n$, we have $\fix\phi\cong F_k$ for some $k\leq n$. Therefore, the conclusion is clear if $n=0,1$ or $m=0$.

Below, we assume that  $n\geq 2$ and $m\geq 1$.
Note that if the endomorphism $\Psi$ is of type (2) in Proposition \ref{DV endo-prop}, then the image of $\Psi$ is abelian, and hence $\fix\Psi$ is abelian. Therefore, $\fix\Psi$ is not isomorphic to $F_{\aleph_0}, ~F_g\times \mathbb{Z}^m\ (g> n)$.
Now we assume that $\Psi$ is of type (1), namely,
$$\Psi=\Psi_{\phi,\textbf{Q}, \textbf{P}}: u\textbf{t}^\textbf{a} \mapsto \phi(u)\textbf{t}^{\textbf{aQ}+\textbf{uP}}$$
where $\phi\in \edo(F_n), \textbf{Q}\in \mathcal{M}_m(\mathbb{Z})$ and $\textbf{P}\in \mathcal{M}_{n\times m}(\mathbb{Z})$.
Then, we have
\begin{eqnarray}\label{eq. fix Psi}
\fix\Psi=\{u\textbf{t}^\textbf{a}\mid\textbf{a}=\textbf{aQ}+\textbf{uP}, ~u=\phi(u)\}.
\end{eqnarray}

Now we consider the projection $$p:  F_n\times \Z^m\to  F_n, \quad u\textbf{t}^\textbf{a}\mapsto u,$$
and let $p_\Psi: \fix\Psi\to p(\fix\Psi)\leq F_n$ be the restriction of $p$ on $\fix\Psi$.
Then we have the natural short exact sequence
$$0\to \ker(p_\Psi)\hookrightarrow\fix\Psi \xrightarrow{p_\Psi} p(\fix\Psi)\to 1,$$
where
\begin{equation}\label{eq 5.5}
p(\fix\Psi)=\{u\in \fix\phi \mid \exists \textbf{a}=\textbf{aQ}+\textbf{uP} \}\vartriangleleft \fix\phi\cong F_k (k\leq n),
\end{equation}
is a normal subgroup of  $\fix\phi$, and
\begin{equation}\label{equ 6}
\ker(p_\Psi)=\{\textbf{t}^\textbf{a}\in \Z^m\mid\textbf{a}=\textbf{aQ}\}\cong \Z^s, ~s\leq m.
\end{equation}
Since $p(\fix\Psi)$ is free, we can define a monomorphism $\iota:p(\fix\Psi)\to \fix\Psi$
sending each element of a chosen free basis for $p(\fix\Psi)$ back to an arbitrary preimage, i.e., $p\comp \iota$ is the identity of $p(\fix\Psi)$. It implies  that
the above exact sequence is split. Note that $\ker(p_\Psi)\cong \Z^s$ is abelian, by straightforward calculations, we can see the following
map is an isomorphism:
\begin{eqnarray}
\Psi: \fix\Psi &\to& \iota(p(\fix\Psi))\times \ker(p_\Psi)\notag\\
 h&\mapsto& \big((\iota\comp p)(h), ~~h\cdot(\iota\comp p)(h^{-1})\big)\notag.
\end{eqnarray}
Therefore,
\begin{equation}\label{eq 5}
\fix\Psi=\iota(p(\fix\Psi))\times \ker(p_\Psi)\cong p(\fix\Psi)\times\Z^s (s\leq m).
\end{equation}

Note that $p(\fix\Psi)$ is free by Eq. (\ref{eq 5.5}), and hence the conclusion of Theorem \ref{Thm 3.1} is clear when $0<s<m$. Below we consider the remained two cases when $s=m$ or $0$.

Case ($s=m$): $\fix\Psi\cong p(\fix\Psi) \times \mathbb{Z}^m.$
By Eq. (\ref{equ 6}) and Eq. (\ref{eq 5}),
$$\ker(p_\Psi)=\{\textbf{t}^\textbf{a}\in \Z^m\mid\textbf{a}=\textbf{aQ}\}=\Z^m,$$
and hence \textbf{Q} = \textbf{I}. Thus, Eq. (\ref{eq 5.5}) becomes
$$p(\fix\Psi)=\{u\in \fix\phi \mid \textbf{uP}=\textbf{0} \},$$
which is the kernel of the epimorphism
$$\gamma: \fix\phi\to \gamma(\fix\phi)\leq \mathbb{Z}^m, \quad u\mapsto \textbf{uP},$$
and hence a normal group of $\fix\phi\cong F_k$ for $k\leq n$ with index
$$[\fix\phi: p(\fix\Psi)]=[F_k:\ker\gamma]=|\gamma(\fix\phi)|=1 ~or ~\aleph_0.$$
Therefore, $p(\fix\Psi)\cong F_k (k\leq n)$, $F_{\aleph_0},$ or trivial by Lemma \ref{F_2 subgp}, and hence
$$\fix\Psi\not\cong F_{\aleph_0}, ~F_g\times\mathbb{Z}^m (g> n).$$
		
Case ($s=0$): $\fix\Psi\cong p(\fix\Psi).$  In this case,
$$\ker(p_\Psi)=\{\textbf{t}^\textbf{a}\in \Z^m\mid\textbf{a}(\textbf{I}-\textbf{Q})=\textbf{0}\}=\textbf{0}.$$
Thus the integer matrix $\textbf{I}-\textbf{Q}$ has an inverse matrix $(\textbf{I}-\textbf{Q})^{-1}\in \mathcal{M}_m(\mathbb{\Q})$, and $d(\textbf{I}-\textbf{Q})^{-1}\in \mathcal{M}_m(\mathbb{Z})$ is an integer matrix for $d=\det(\textbf{I}-\textbf{Q})$. It implies that, for every $u\in \fix\phi$, the equation
$$\textbf{a}=\textbf{aQ}+d\textbf{uP}$$
in Eq. (\ref{eq 5.5}) always has a unique solution
$$\textbf{a}=d\textbf{uP}(\textbf{I}-\textbf{Q})^{-1}=\textbf{uP}(d(\textbf{I}-\textbf{Q})^{-1})\in \Z^m.$$
Note that $d\textbf{u}$ is the abelianization of $u^d\in \fix\phi$. Therefore, by Eq. (\ref{eq 5.5}), the subgroup
$$\Gamma_d:=\langle u^d,[u,v]\mid u,v\in\fix\phi\rangle \subset p(\fix\Psi)$$
has index
$$[\fix\phi : p(\fix\Psi)]\leq[\fix\phi : \Gamma_d]\leq d^k,k=\rk(\fix\phi).$$
Recall that $\fix\phi\cong F_k$ for some $k\leq n$, and hence $p(\fix\Psi)$ is a free group of finite rank. So
$$\fix\Psi\cong p(\fix\Psi)\not\cong F_{\aleph_0}, ~F_g\times\mathbb{Z}^m (g> n).$$
The proof is finished.
\end{proof}

\subsection{Proofs of Theorem \ref{main thm 1-free} and Theorem \ref{main thm 3-Z free}}

\begin{proof}[\textbf{Proof of Theorem \ref{main thm 1-free}}]
Here, Item (3) and Item (4) are equivalent by Lemma \ref{DV subgp}. Item (1) implies Item (2) clearly, Item (2) implies Item (3) by Theorem \ref{Thm 3.1}, and \cite[Theorem 1.3]{LWZ25} showed that Item (1) and Item (3) are equivalent. So the three items are all equivalent.
\end{proof}

Now we consider the special case $F_n \times \mathbb{Z}(n \geq 2)$. The structure of its aut-fixed subgroups has been studied in \cite[Theorem 4.3]{LWZ25}.

\begin{thm}[Lei-Wang-Zhang, \cite{LWZ25}]\label{LWZ aut-fixed free-Z}
	A subgroup of $F_n \times \mathbb{Z}(n \geq 2)$ is aut-fixed up to isomorphism if and only if it has one of the following forms:
	$$F_{2t-1}(1 \leq t \leq n),\ \ \ F_t \times \mathbb{Z}^s(0 \leq t \leq n, s=0,1),\ \ \ or \ \ \ F_{\aleph_0} \times \mathbb{Z}.$$
\end{thm}

However, there are completely different conclusions for end-fixed subgroups in this case, see Theorem \ref{main thm 3-Z free}. To prove it, we will find some subgroups which are end-fixed but not aut-fixed up to isomorphism.

\begin{proof}[\textbf{Proof of Theorem \ref{main thm 3-Z free}}]
By Lemma \ref{DV subgp}, every subgroup of $F_n\times\mathbb{Z} (n \geq 2)$ has one of the following forms:
$$F_{\aleph_0}, ~F_{\aleph_0}\times\mathbb{Z}, ~ F_t, ~F_t\times\mathbb{Z}  ~(t\geq 0).$$
Then Item (1) is equivalent to Theorem \ref{LWZ aut-fixed free-Z}.

Now we prove Item (2). Note that the subgroups $F_{\aleph_0}$ and $F_t\times\mathbb{Z} (t> n)$ are not end-fixed in $F_n \times \mathbb{Z}$ by Theorem \ref{Thm 3.1}, while
$F_t, ~F_t\times \mathbb{Z}(0 \leq t \leq n)$ and $F_{\aleph_0} \times \mathbb{Z}$ are aut-fixed (hence end-fixed) in $F_n \times \mathbb{Z}$ by Theorem \ref{LWZ aut-fixed free-Z}.

Below we show that every subgroup $F_t (t>n)$ is also end-fixed. For any $t>n$, let $\Psi: F_n\times\mathbb{Z}\to F_n\times\mathbb{Z}$ defined as
$$(u, v)\mapsto (\phi(u), tv+\gamma(u)),$$
where $\phi:F_n \to F_n=\langle x_1, x_2,\ldots, x_n \rangle$ is defined as
$$x_1 \mapsto x_1, ~x_2 \mapsto x_2, ~x_i \mapsto x_i^{-1}, ~i>2,$$
and $\gamma(u)$ is the sum of powers of $x_1$ in $u$.
Then $\fix\phi=\langle x_1, x_2\rangle\cong F_2$, and
\begin{eqnarray}
\fix\Psi &=& \{(u, v)\in F_n\times\mathbb{Z} \mid u\in\fix\phi, ~v=\gamma(u)/(1-t)\}\notag\\
&\cong& \{u\in\fix\phi\mid \gamma(u)\equiv 0 \mod t-1\}\notag,
\end{eqnarray}
which is a subgroup of $\fix\phi\cong F_2$ with index $t-1$. So, $\fix\Psi \cong F_{t}(t>n)$.
\end{proof}

Clearly, we have a direct corollary of Theorem \ref{main thm 3-Z free} as follows:

\begin{cor}
A subgroup $H$ of $F_n\times\Z(n \geq 2)$, is end-fixed but not aut-fixed up to isomorphism if and only if
 $$H\cong F_{2k} (2k> n ), ~F_t(t>2n).$$
\end{cor}

\section{Fixed subgroups in $\pi_1(\Sigma_g)\times \Z^m$}\label{sect 4}

In this section, we will deal with the case of surface groups and prove Theorem \ref{main thm 2-surf} and Theorem \ref{main thm 4-Z surf}. Let $G=\pi_1(\Sigma_g)\times\mathbb{Z}^m (g\geq 2, m\geq 1)$ with the following representation
$$\left<x_1, x_2,\dots, x_{2g-1}, x_{2g},t_1,\dots,t_m\mid[x_1,x_2]\cdots[x_{2g-1},x_{2g}],[t_i,x_j],[t_i,t_j]\right>.$$
To be consistent with the previous, we use the same form to represent each element of $G$:
$$u\textbf{t}^\textbf{a}=u(x_1,\dots,x_{2g})t_1^{a_1}\cdots t_m^{a_m}$$
where $\textbf{a}=(a_1,\dots,a_m)\in \mathbb{Z}^m$, is a row vector, and $u =u(x_1, \dots , x_{2g})$ is a freely
reduced word on the alphabet $X=\{x_1,\dots, x_{2g}\}$, which has the shortest length (note that the form of this word may be not unique).

\subsection{Endomorphisms}
We get a parallel characterization of endomorphisms of $ \pi_1(\Sigma_g)\times\mathbb{Z}^m$ as in Proposition \ref{DV endo-prop}.

\begin{prop}\label{des of surfend}
	Let $G=\pi_1(\Sigma_g)\times\mathbb{Z}^m(g\geq 2, m\geq 1)$ with a presentation given at the beginning of this section. Then each endomorphism $\Psi \in \edo(G)$ can be represented as one of the following two forms:
	\begin{enumerate}
		\item $ \Psi_{\phi,\textbf{Q},\textbf{P}}=u\textbf{t}^\textbf{a} \mapsto \phi(u)\textbf{t}^{\textbf{aQ}+\textbf{uP}} $, where $ \phi\in \edo(\pi_1(\Sigma_g)), \textbf{Q}\in \mathcal{M}_m(\mathbb{Z}) $ and $ \textbf{P}\in \mathcal{M}_{2g\times m}(\mathbb{Z})$.
		
		\item$\Psi_{z,\textbf{l,h,Q,P}}=u\textbf{t}^\textbf{a}\mapsto z^{\textbf{al}^T+\textbf{uh}^T}\textbf{t}^{\textbf{aQ}+\textbf{uP}}$, where $1\ne z\in \pi_1(\Sigma_g)$ is not a proper power, $\textbf{Q}\in \mathcal{M}_m(\mathbb{Z}) $, $ \textbf{P}\in \mathcal{M}_{2g\times m}(\mathbb{Z}),0\ne \textbf{l} \in \mathbb{Z}^m$, and $\textbf{h}\in \mathbb{Z}^{2g}$.
	\end{enumerate}
	In both cases, $\textbf{u}\in \mathbb{Z}^{2g}$ denotes the abelianization of the word $ u \in \pi_1(\Sigma_g) $.
\end{prop}
\begin{proof}
	Note that we only need to consider all the images of generators under the endomorphism $\Psi$. Now, we set
	\begin{eqnarray*}
		\Psi(x_i)&=&c_i\textbf{t}^{\textbf{p}_{i}}, ~1\leq i\leq 2g,\\
        \Psi(t_i)&=&z_i\textbf{t}^{\textbf{q}_{i}},  ~1\leq i\leq m.
	\end{eqnarray*}

(1) If all the $z_i$'s are trivial, namely
        $$\Psi(t_i)=\textbf{t}^{\textbf{q}_{i}}, \quad 1\leq i\leq m.$$
Let $p:\pi_1(\Sigma_g)\times\mathbb{Z}^m\to \pi_1(\Sigma_g), ~u\textbf{t}^{\textbf{a}}\mapsto u,$ and $\phi=(p\comp\Psi)|_{\pi_1(\Sigma_g)}$. Then  $\phi$ can be seen as an endomorphism of $\pi_1(\Sigma_g)$, and
	$$ \Psi=\Psi_{\phi,\textbf{Q},\textbf{P}}: u\textbf{t}^\textbf{a} \mapsto \phi(u)\textbf{t}^{\textbf{aQ}+\textbf{uP}}, \quad\forall u\textbf{t}^\textbf{a}\in \pi_1(\Sigma_g)\times\mathbb{Z}^m,$$
	where $\textbf{Q}=(\textbf{q}_{1},\dots,\textbf{q}_{m})^T\in \mathcal{M}_m(\mathbb{Z})$,
 $\textbf{P}=(\textbf{p}_{1},\dots,\textbf{p}_{2g})^T\in \mathcal{M}_{2g\times m}(\mathbb{Z})$, and $\textbf{u}$ denotes the abelianization of $u$.
	
(2) If some $z_i\neq 1$. Since $t_i$ commutes with all the $t_j$'s and $x_k$'s, we have
	$$z_iz_j=z_jz_i, \quad z_ic_k=c_kz_i.$$
Then, there exists an element $z\in \pi_1(\Sigma_g)$ (not a proper power) such that
$$1\neq \left\langle z_1,\dots, z_m, c_1, \ldots, c_{2g}\right\rangle \subset \left\langle z \right\rangle.$$ (Two nontrivial elements of $\pi_1(\Sigma_g)(g\geq 2)$ commute if and only if they share the same axis in $\Sigma_g$'s universal covering space: the Poincar\'{e} disk).
	With a similar process as in the above case (1), we can get
	$$\Psi=\Psi_{z,\textbf{l,h,Q,P}}: u\textbf{t}^\textbf{a}\mapsto z^{\textbf{al}^T+\textbf{uh}^T}\textbf{t}^{\textbf{aQ}+\textbf{uP}},$$
	where $1\ne z\in \pi_1(\Sigma_g)$ , $\textbf{Q}\in \mathcal{M}_m(\mathbb{Z}),$ $ \textbf{P}\in \mathcal{M}_{2g\times m}(\mathbb{Z})$, $0\ne \textbf{l} \in \mathbb{Z}^m$, and $\textbf{h}\in \mathbb{Z}^{2g}$.
\end{proof}

Since the forms of endomorphisms of $\pi_1(\Sigma_g)\times\mathbb{Z}^m$ are similar to that of $F_n \times \Z^m$, we can also give the following characterization of monomorphisms and epimorphisms as Delgado and Ventura did in \cite[Proposition 5.2]{DV13}.

\begin{prop}\label{Hopf non co-hopf prop}
	Let $G=\pi_1(\Sigma_g)\times\mathbb{Z}^m (g\geq 2, m\geq 1)$ and $\Psi \in \edo(G)$. Then,
	\begin{enumerate}
		\item $\Psi$ is a monomorphism if and only if $\Psi$ is the first kind of endomorphism $\Psi_{\phi,\textbf{Q},\textbf{P}}$, with $\phi$ a monomorphism of $\pi_1(\Sigma_g)$ and $\mathrm{det}(\textbf{Q}) \neq 0$.
		
		\item $\Psi$ is an epimorphism if and only if $\Psi$ is the first kind of endomorphism $\Psi_{\phi, \textbf{Q}, \textbf{P}}$, with $\phi$ an epimorphism of $\pi_1(\Sigma_g)$ and $\mathrm{det}(\textbf{Q}) = \pm 1$.
		
		\item $\Psi$ is an automorphism if and only if $\Psi$ is an epimorphism. Hence $G$ is Hopfian but not co-Hopfian and $$(\Psi_{\phi,\textbf{Q},\textbf{P}})^{-1}=\Psi_{\phi^{-1},\textbf{Q}^{-1}, -\textbf{M}^{-1} \textbf{P} \textbf{Q}^{-1}},$$ where $\textbf{M} \in \mathrm{GL}_{2g}(\Z)$ is the abelianization of $\phi$.
	\end{enumerate}
\end{prop}

\begin{proof}
	(1) $``\Rightarrow":$ Suppose $\Psi$ is a monomorphism. Then $G\cong \im \Psi $ is not abelian. Note that the image $\im{\Psi_{z,\textbf{l,h,Q,P}}}$ of the second type endomorphism is abelian, so $\Psi$ must be of the first type, namely, $\Psi=\Psi_{\phi,\textbf{Q},\textbf{P}}$ for some $\phi\in \edo(\pi_1(\Sigma_g)), \textbf{Q}\in \mathcal{M}_m(\mathbb{Z}) $ and $ \textbf{P}\in \mathcal{M}_{2g\times m}(\mathbb{Z})$.
Since $\Psi(\textbf{t}^\textbf{a})=\textbf{t}^{\textbf{aQ}}$, it is easy to see $\det(\textbf{Q})\neq 0$ by the injectivity of $\Psi$.
Finally, we prove the injectivity of $\phi$. Suppose $\phi(u)=1$. Then $\textbf{Q}^{-1}\det(\textbf{Q})\in \mathcal{M}_m(\mathbb{Z})$ and
$$\Psi(u^{-\mathrm{det}(\textbf{Q})} \textbf{t}^{ \textbf{uP}\textbf{Q}^{-1}\mathrm{det}(\textbf{Q})})=\phi(u^{-\mathrm{det}(\textbf{Q})}) \textbf{t}^{\textbf{uP}\textbf{Q}^{-1}\det(\textbf{Q})\cdot\textbf{Q}+(-\det(\textbf{Q})\textbf{u})\textbf{P}}=1.$$
So the injectivity of $\Psi$ implies $u^{-\mathrm{det}(\textbf{Q})}=1$, and hence $u=1$ because
$\pi_1(\Sigma_g)$ is torsion free. Therefore, $\phi$ is a monomorphism.
	
	$``\Leftarrow":$ Suppose $\Psi=\Psi_{\phi,\textbf{Q},\textbf{P}}$ is an endomorphism of the first type, $\phi \in \edo(\pi_1(\Sigma_g))$ is injective and $\det(\textbf{Q}) \neq 0$.
	Let $u \textbf{t}^\textbf{a} \in G$ and $\Psi(u \textbf{t}^\textbf{a})=\phi(u) \textbf{t}^{\textbf{aQ}+\textbf{uP}}=1$. Then the injectivity of $\phi$ implies $u=1$, so we have
	$\textbf{aQ}+\textbf{uP}=\textbf{aQ}=\textbf{0}$, hence $\textbf{a}=\textbf{0}$ and $\Psi$ is injective.
	
	(2) $``\Rightarrow":$ Suppose $\Psi$ is an epimorphism. Then $\Psi$ must be the first type endomorphism, because $\im{\Psi_{z,\textbf{l,h,Q,P}}}$ is contained in a proper subgroup of $G$.
	For any $v \in \pi_1(\Sigma_g) \leq G$, suppose $\Psi(u \textbf{t}^\textbf{a})=v$, which implies that $\phi(u)=v$, and $\phi$ is an epimorphism. Finally we prove
	$\det (\textbf{Q})=\pm 1$. Suppose
$$\Psi(u_j \textbf{t}^{\textbf{b}_\textbf{j}})=\phi(u_j) \textbf{t}^ { \textbf{b}_\textbf{j}\textbf{Q}+\textbf{u}_\textbf{j}\textbf{P} }=\textbf{t}^{\bm{\delta}_j},$$
where $j \in \{1,2,\ldots,m\}$ and
	$\bm{\delta}_j=(0,\cdots,0,1,0,\cdots,0) \in \Z^m$ with the $j$-th coordinate 1. Then we have $\phi(u_j)=1$, so $u_j=1$ because $\pi_1(\Sigma_g)$ is Hopfian and $\phi$ is an epimorphism,
	this means $\textbf{b}_\textbf{j}\textbf{Q}=\textbf{b}_\textbf{j}\textbf{Q}+\textbf{u}_\textbf{j}\textbf{P}=\bm{\delta}_j$ for any $j \in \{1,2,\ldots,m\}$, so there is a matrix $\textbf{B}$ such that $\textbf{BQ}=\textbf{I}_m$, hence $\det(\textbf{Q})= \pm 1$.

	$``\Leftarrow":$ Suppose $\Psi=\Psi_{\phi,\textbf{Q},\textbf{P}}$ is an endomorphism of the first type, $\phi \in \edo(\pi_1(\Sigma_g))$ is surjective and $\det(\textbf{Q})=\pm 1$.
	Notice that $\pi_1(\Sigma_g)$ is Hopfian and then $\phi \in \aut(\pi_1(\Sigma_g))$. We denote $\Psi'=\Psi_{\phi^{-1},\textbf{Q}^{-1}, -\textbf{M}^{-1} \textbf{P} \textbf{Q}^{-1}}$, where $\textbf{M} \in \mathrm{GL}_{2g}(\Z)$ is the abelianization of $\phi$.
	For any $u \textbf{t}^\textbf{a} \in G$,
	\begin{eqnarray}\label{equ epi of G}
		\Psi \circ \Psi'(u \textbf{t}^\textbf{a}) &=& \Psi(\phi^{-1}(u) \textbf{t}^{\textbf{aQ}^{-1}-\textbf{uM}^{-1}\textbf{PQ}^{-1}})\\
		&=& u \textbf{t}^{\textbf{a}-\textbf{uM}^{-1}\textbf{P}+\textbf{uM}^{-1}\textbf{P}}\notag\\
		&=& u \textbf{t}^\textbf{a}.\notag
	\end{eqnarray}
	Then $\Psi$ is surjective.
	
	(3)  It suffices to prove that any epimorphism $\Psi$ of $G$ is an automorphism. Note that $\Psi=\Psi_{\phi,\textbf{Q},\textbf{P}}$ by the above Item (2). For any $u \textbf{t}^\textbf{a} \in G$, we have
	\begin{eqnarray}
		\Psi' \comp \Psi(u \textbf{t}^\textbf{a}) &=& \Psi'(\phi(u)\textbf{t}^{\textbf{aQ}+\textbf{uP}} )\\
		&=& u \textbf{t}^{\textbf{a}+\textbf{uP}\textbf{Q}^{-1}-\textbf{uP}\textbf{Q}^{-1}}\notag\\
		&=& u \textbf{t}^\textbf{a}.\notag
	\end{eqnarray}
	By combining Eq. (\ref{equ epi of G}), we obtain that $\Psi=\Psi_{\phi,\textbf{Q},\textbf{P}}$ is an automorphism with
$$\Psi^{-1}=\Psi'=\Psi_{\phi^{-1},\textbf{Q}^{-1}, -\textbf{M}^{-1} \textbf{P} \textbf{Q}^{-1}}.$$
Hence $G$ is Hopfian. Moreover, from the proof of (1), it is clear that a monomorphism may not necessarily be an epimorphism, so $G$ is not co-Hopfian.
\end{proof}

\begin{proof}[\textbf{Proof of Corollary \ref{Hopf cor}}]
Let $G=H\times\mathbb{Z}^m (m\geq 1)$ for $H=F_g$ or $H=\pi_1(\Sigma_g)$ $(g\geq 1)$. We have three cases:

Case (1). If $H=F_1\cong \Z$, then $G=\Z^{m+1}$ is Hopfian and not co-Hopfian as a finitely generated free abelian group. If $H=F_g$ with $g\geq 2$, then $G$ is a direct product of residually finite groups and hence $G$ is residually finite, so $G$ is Hopfian because every finitely generated residually finite group is already Hopfian. Let $F_g$ be freely generated by $\{a_1,a_2,\dots,a_g\}$ and denote $A=\left<a_1\right>$ and $B=\left<a_2,\dots,a_g\right>$. Then, $G=(A*B)\times\Z$ is not co-Hopfian since $G$ is isomorphic to its proper subgroup $(A*(a_2a_1Ba_1^{-1}a_2^{-1}))\times\Z$.

Case (2). If $H=\pi_1(\Sigma_1)=\Z^2$, then $G=\Z^{m+2}$ is Hopfian but not co-Hopfian.

Case (3). If $H=\pi_1(\Sigma_g)(g\geq 2)$, then $G$ is Hopfian but not co-Hopfian by Item (3) in Proposition \ref{Hopf non co-hopf prop}.
	\end{proof}

\begin{proof}[\textbf{Proof of Corollary \ref{Hopf cor for hyperbolic case}}]

Since all hyperbolic groups are Hopfian \cite{FS23}, we have $H$ is Hopfian.

Case (1). $H$ is non-elementary. Note that every non-elementary hyperbolic group has finite center. Therefore, the center $Z(H)$ is finite. Then, the quotient group $H/Z(H)$ is also hyperbolic and thus Hopfian.
Since $H$ and $H/Z(H)$ are Hopfian and $\Z^m$ is finitely generated abelian, $H\times\Z^m$ is Hopfian according to a corollary in \cite{Hi69}.

Case (2). $H$ is elementary, i.e., finite or virtually infinite cyclic. In this case, $H$ is residually finite. Then, the direct product $H\times\Z^m$ is also residually finite and thus Hopfian.

Note that $H\times(2\Z)^m$ is a proper subgroup of $H\times\Z^m$ and $H\times(2\Z)^m\cong H\times\Z^m$. Hence, $H\times\Z^m$ is not co-Hopfian.
\end{proof}

\subsection{Fixed subgroups}
Now, we have the following.

\begin{thm}\label{thm for surf>2}
	Let $\Psi$ be an endomorphism of $ G =\pi_1(\Sigma_g)\times\mathbb{Z}^m  (g \geq 2) $. Then $\fix\Psi$ is not isomorphic to
	$$ F_{\aleph_0}, ~F_t\times \mathbb{Z}^m (t \geq 2g), ~\pi_1(\Sigma_n) \times \mathbb{Z}^m(n > g). $$	
\end{thm}
\begin{proof}
If $m=0$, then $ G =\pi_1(\Sigma_g)(g \geq 2) $ and hence $\fix\Psi \cong F_k (k<2g)$ or $\pi_1(\Sigma_g)$ by \cite{JWZ11}. So the conclusion holds.

Below we assume $m\geq 1$. Note that if the endomorphism $\Psi$ is of type (2) in Proposition \ref{des of surfend}, then the image of $\Psi$ is abelian, and hence $ \fix\Psi $ is abelian. Therefore, $ \fix\Psi $ is not isomorphic to $F_{\aleph_0}, F_t\times \mathbb{Z}^m (t \geq 2g), \pi_1(\Sigma_n) \times \mathbb{Z}^m(n > g)$. Now we assume that $\Psi$ is of type (1), namely,
	$$\Psi=\Psi_{\phi,\textbf{Q}, \textbf{P}}: u\textbf{t}^\textbf{a} \mapsto \phi(u)\textbf{t}^{\textbf{aQ}+\textbf{uP}},$$
	where $\phi\in \edo(\pi_1(\Sigma_g)), \textbf{Q}\in \mathcal{M}_m(\mathbb{Z})$ and $\textbf{P}\in \mathcal{M}_{2g\times m}(\mathbb{Z})$.
	Then, we have
	\begin{eqnarray}\label{eq. fix Psi surf}
		\fix\Psi=\{u\textbf{t}^\textbf{a}\mid\textbf{a}=\textbf{aQ}+\textbf{uP}, ~u=\phi(u)\}.
	\end{eqnarray}
	
	Now we consider the projection $$p:  \pi_1(\Sigma_g)\times \Z^m\to  \pi_1(\Sigma_g), \quad u\textbf{t}^\textbf{a}\mapsto u,$$
	and let $p_\Psi: \fix\Psi\to p(\fix\Psi)\leq \pi_1(\Sigma_g)$ be the restriction of $p$ on $\fix\Psi$.
	Then we have the natural short exact sequence
	\begin{equation}\label{split exact for surf}
		0\to \ker(p_\Psi)\hookrightarrow\fix\Psi \xrightarrow{p_\Psi} p(\fix\Psi)\to 1,
	\end{equation}
	where
	\begin{eqnarray}\label{eq s5.5}
		p(\fix\Psi)=\{u\in \fix\phi \mid \exists \textbf{a}=\textbf{aQ}+\textbf{uP} \}\vartriangleleft \fix\phi,
	\end{eqnarray}
	is a normal subgroup of $\fix\phi\cong F_k (k<2g)\mbox{ or }\pi_1(\Sigma_g)$, and
	\begin{equation}\label{equ s6}
		\ker(p_\Psi)=\{\textbf{t}^\textbf{a}\in \Z^m\mid\textbf{a}=\textbf{aQ}\}\cong \Z^s ~(s\leq m).
	\end{equation}
	
	Since the exact sequence Eq. (\ref{split exact for surf}) is split (see the proof of \cite[Lemma 2.4]{LWZ25}), we may get a monomorphism $\iota:p(\fix\Psi)\to \fix\Psi$. Note that $\ker(p_\Psi)\cong \Z^s$ is abelian, by straightforward calculations, we can see the following
	map is an isomorphism:
	\begin{eqnarray}
		\Psi: \fix\Psi &\to& \iota(p(\fix\Psi))\times \ker(p_\Psi)\notag\\
		h&\mapsto& \big((\iota\comp p)(h), ~~h\cdot(\iota\comp p)(h^{-1})\big)\notag.
	\end{eqnarray}
	Therefore,
	\begin{equation}\label{eq s5}
		\fix\Psi=\iota(p(\fix\Psi))\times \ker(p_\Psi)\cong p(\fix\Psi)\times\Z^s (s\leq m).
	\end{equation}
	
Recall that $p(\fix\Psi)$ is a free group or a surface group by Eq. (\ref{eq s5.5}). Hence the conclusion of Theorem \ref{thm for surf>2} is clear when $0<s<m$. Below we only need to consider the cases when $s=m$ or $0$.\\
	
Case ($s=m$): $\fix\Psi\cong p(\fix\Psi) \times \mathbb{Z}^m.$
	By Eq. (\ref{equ s6}) and Eq. (\ref{eq s5}),
	$$\ker(p_\Psi)=\{\textbf{t}^\textbf{a}\in \Z^m\mid\textbf{a}=\textbf{aQ}\}=\Z^m,$$
	and hence $ \textbf{Q} = \textbf{I} $. Thus, Eq. (\ref{eq s5.5}) becomes
	$$p(\fix\Psi)=\{u\in \fix\phi \mid \textbf{uP}=\textbf{0} \},$$
	which is the kernel of the epimorphism
	$$\gamma: \fix\phi\to \gamma(\fix\phi)\leq \mathbb{Z}^m, \quad u\mapsto \textbf{uP},$$
	and hence a normal subgroup of $\fix\phi\cong F_k(k< 2g)$ or $\pi_1(\Sigma_g)$ with index
	$$[\fix\phi: p(\fix\Psi)]=[F_k:\ker\gamma]=|\gamma(\fix\phi)|=1 \mbox{ or } \aleph_0,$$
	or
	$$[\fix\phi: p(\fix\Psi)]=[\pi_1(\Sigma_g):\ker\gamma]=|\gamma(\fix\phi)|=1 \mbox{ or } \aleph_0.$$
	When $|\gamma(\fix\phi)|=1$, we have $p(\Psi)\cong F_k (k<2g)\mbox{ or }\pi_1(\Sigma_g)$; when $|\gamma(\fix\phi)|=\aleph_0$, we have $p(\Psi)\cong F_{\aleph_0}$ or trivial, by the fact that an infinite index nontrivial normal subgroup of $\pi_1(\Sigma_g)$ is $F_{\aleph_0}$
(see Lemma \ref{F_2 subgp} and Lemma \ref{subgp of surface}), and hence
	$$\fix\Psi\not\cong F_{\aleph_0}, ~F_t\times\mathbb{Z}^m (t\geq 2g),~\pi_1(\Sigma_n) \times \mathbb{Z}^m(n > g).$$
	
Case ($s=0$): $\fix\Psi\cong p(\fix\Psi).$  In this case,
	$$\ker(p_\Psi)=\{\textbf{t}^\textbf{a}\in \Z^m\mid\textbf{a}(\textbf{I}-\textbf{Q})=\textbf{0}\}=\textbf{0}.$$
	Thus the integer matrix $\textbf{I}-\textbf{Q}$ has an inverse matrix $(\textbf{I}-\textbf{Q})^{-1}\in \mathcal{M}_m(\mathbb{\Q})$, and $d(\textbf{I}-\textbf{Q})^{-1}\in \mathcal{M}_m(\mathbb{Z})$ is an integer matrix for $d=\det(\textbf{I}-\textbf{Q})$. It implies that, for every $u\in \fix\phi$, the equation
	$$\textbf{a}=\textbf{aQ}+d\textbf{uP}$$
	in Eq. (\ref{eq s5.5}) always has a unique solution
	$$\textbf{a}=d\textbf{uP}(\textbf{I}-\textbf{Q})^{-1}=\textbf{uP}(d(\textbf{I}-\textbf{Q})^{-1})\in \Z^m.$$
	Note that $d\textbf{u}$ is the abelianization of $u^d\in \fix\phi$. Therefore, by Eq. (\ref{eq s5.5}), the subgroup
	$$\Gamma_d:=\langle u^d,[u,v]\mid u,v\in\fix\phi\rangle \subset p(\fix\Psi)$$
	has index
	$$[\fix\phi : p(\fix\Psi)]\leq[\fix\phi : \Gamma_d]\leq d^k,k=\rk(\fix\phi).$$
	Recall that $\fix\phi\cong \pi_1(\Sigma_g)$ or $F_k$ for some $k< 2g$, and hence $p(\fix\Psi)$ is a free group of finite rank or a surface group. So
	$$\fix\Psi\cong p(\fix\Psi)\not\cong F_{\aleph_0}, ~F_t\times\mathbb{Z}^m (t\geq 2g),~\pi_1(\Sigma_n) \times \mathbb{Z}^m(n > g).$$
The proof is finished.
\end{proof}

\subsection{Proofs of Theorem \ref{main thm 2-surf} and Theorem \ref{main thm 4-Z surf}}

\begin{proof}[\textbf{Proof of Theorem \ref{main thm 2-surf}}]
 Item (1) implies Item (2) clearly, Item (2) implies Item (3) by Theorem \ref{thm for surf>2}. Furthermore, Item (3) and Item (4) are equivalent by Lemma \ref{surface subgp}, and Item (1) and Item (3) are equivalent by \cite[Theorem 1.4]{LWZ25}.
\end{proof}

Now, we consider the group $\pi_1(\Sigma_g)\times\Z$. The aut-fixed subgroups in this case are the following \cite[Theorem 5.3]{LWZ25}.

\begin{thm}[Lei-Wang-Zhang, \cite{LWZ25}]\label{4.8}
	A subgroup of $\pi_1(\Sigma_g)\times \mathbb{Z}$ ($g\geq 2$) is aut-fixed up to isomorphism if and only if it has one of the following forms:
	$$F_{2t-1} ~(1\leq t< 2g), \quad \pi_1(\Sigma_{2g-1}), \quad F_{\aleph_0}\times \mathbb{Z},$$
	$$\pi_1(\Sigma_g)\times \mathbb{Z}^s, \quad F_t\times \mathbb{Z}^s ~(0\leq t<2g, ~s=0,1).$$
\end{thm}

However, if we consider the end-fixed subgroups in $\pi_1(\Sigma_g)\times\Z$, we will obtain a result similar to Theorem \ref{main thm 3-Z free} but different from Theorem \ref{4.8}, that is, comparing to automorphisms, it contains more end-fixed subgroups, see Theorem \ref{main thm 4-Z surf}.

\begin{proof}[\textbf{Proof of Theorem \ref{main thm 4-Z surf}}]
	By Lemma \ref{surface subgp}, every subgroup of $\pi_1(\Sigma_g)\times\mathbb{Z} (g \geq 2)$ has one of the following forms:
	$$F_{\aleph_0}\times\mathbb{Z}^s(s=0,1), ~F_t\times\mathbb{Z}^s(t \geq 0, ~s=0,1), $$
	$$\pi_1(\Sigma_{m(g-1)+1})\times\mathbb{Z}^s(m \geq 1, ~s=0,1).$$
Then Item (1) is equivalent to Theorem \ref{4.8}.

Now we prove Item (2). Note that the subgroups $F_{\aleph_0}$, $F_t\times\mathbb{Z} (t \geq 2g)$ and $\pi_1(\Sigma_n)\times\mathbb{Z} (n>g)$ are not end-fixed in $\pi_1(\Sigma_g) \times \mathbb{Z}$ by Theorem \ref{thm for surf>2}, while
	$$\pi_1(\Sigma_{2g-1}), ~F_{\aleph_0} \times \mathbb{Z}, ~\pi_1(\Sigma_{g})\times\mathbb{Z}^s (s=0,1), ~F_t\times \mathbb{Z}^s(0 \leq t < 2g, s=0,1),$$
	are aut-fixed (hence end-fixed) in $\pi_1(\Sigma_g) \times \mathbb{Z}$ by Theorem \ref{4.8}.
	
	Below we show that every subgroup $F_t (t \geq 2g)$ and $\pi_1(\Sigma_{m(g-1)+1}) (m>2)$ are also end-fixed.

For any $t \geq 2g$, let $\Psi: \pi_1(\Sigma_g)\times\mathbb{Z}\to \pi_1(\Sigma_g)\times\mathbb{Z}$ defined as
	$$(u, v)\mapsto (\phi(u), tv+\gamma(u)),$$ 	
	where $\phi \in \edo(\pi_1(\Sigma_g))$ such that $\fix \phi=\langle x_1,x_2\rangle \cong F_2$ (such $\phi$ is constructible, see \cite[Proposition 2.9]{LWZ25}) contained in
$$\pi_1(\Sigma_g)=\left<x_1, x_2,\dots, x_{2g-1}, x_{2g}\mid[x_1,x_2]\cdots[x_{2g-1},x_{2g}]\right>,$$
	and $\gamma(u)$ is the sum of powers of $x_1$ in $u$.
	Then
	\begin{eqnarray}
		\fix\Psi &=& \{(u, v)\in \pi_1(\Sigma_g)\times\mathbb{Z} \mid u\in\fix\phi, ~v=\gamma(u)/(1-t)\}\notag\\
		&\cong& \{u\in\fix\phi\mid \gamma(u)\equiv 0 \mod t-1\}\notag,
	\end{eqnarray}
	which is a subgroup of $\fix\phi \cong F_2$ with index $t-1$. So, $\fix\Psi \cong F_{t}$.
	
	For any $m>2$, let $\Psi: \pi_1(\Sigma_g)\times\mathbb{Z}\to \pi_1(\Sigma_g)\times\mathbb{Z}$ defined as
	$$(u, v)\mapsto (u, (1-m)v+\gamma(u)),$$
	where $\gamma(u)$ as above. Then
	\begin{eqnarray}
		\fix\Psi &=& \{(u, v)\in \pi_1(\Sigma_g)\times\mathbb{Z} \mid v=\gamma(u)/m\}\notag\\
		&\cong& \{u\in \pi_1(\Sigma_g)\mid \gamma(u)\equiv 0 \mod m\}\notag,
	\end{eqnarray}
	which is a subgroup of $\pi_1(\Sigma_g)$ with index $m$. Thus, $\fix\Psi \cong \pi_1(\Sigma_{m(g-1)+1})$ by Lemma \ref{subgp of surface}.
\end{proof}

Finally, we have a direct corollary of Theorem \ref{main thm 4-Z surf} as follows:

\begin{cor}
A subgroup $H$ of $\pi_1(\Sigma_g)\times\Z(g \geq 2)$, is end-fixed but not aut-fixed up to isomorphism, if and only if
 $$H \cong F_{2n} (n\geq g ), ~F_t(t \geq 4g-1), ~~\pi_1(\Sigma_k)(2g-1\neq k>g).$$
\end{cor}

\section{Fixed subgroups in free-abelian times hyperbolic groups}\label{sect 5}

In this section, we consider endomorphisms of $H\times\Z^m$ for $H$ a hyperbolic group. First, we have an observation on torsion-free hyperbolic groups which can be found in some textbooks, for example \cite{BH99}.

\begin{lem}\label{lem coummte same axis}
    Let $H$ be a torsion-free hyperbolic group. If two nontrivial elements $x, y\in H$ commute, then there exists a nontrivial element $z\in H$ such that the subgroups $\langle x,y\rangle=\langle z\rangle$.
\end{lem}

\subsection{Endomorphisms}
By applying a similar argument as in Proposition \ref{des of surfend}, we can obtain the following analogous result for endomorphisms of $H \times \Z^m$.

\begin{prop}\label{des of hyperbolic-end}
Let $G=H\times\mathbb{Z}^m (m\geq 1)$, where $H$ is a non-elementary torsion-free hyperbolic group with the first betti number $\beta_1$. Then each endomorphism $\Psi \in \edo(G)$ can be represented as one of the following two forms:
	\begin{enumerate}
		\item $\Psi_{\phi,\textbf{Q},\textbf{P}}=u\textbf{t}^\textbf{a} \mapsto \phi(u)\textbf{t}^{\textbf{aQ}+\textbf{uP}} $, where $ \phi\in \edo(H), \textbf{Q}\in \mathcal{M}_m(\mathbb{Z}) $ and $ \textbf{P}\in \mathcal{M}_{\beta_1\times m}(\mathbb{Z})$.

        \item $\Psi_{z,\textbf{l,h,Q,P}}=u\textbf{t}^\textbf{a}\mapsto z^{\textbf{al}^T+\textbf{uh}^T}\textbf{t}^{\textbf{aQ}+\textbf{uP}}$, where $1\ne z\in H$ is not a proper power, $\textbf{Q}\in \mathcal{M}_m(\mathbb{Z}) $, $ \textbf{P}\in \mathcal{M}_{\beta_1\times m}(\mathbb{Z}),\textbf{0}\ne \textbf{l} \in \mathbb{Z}^m$, and $\textbf{h}\in \mathbb{Z}^{\beta_1}$.
	\end{enumerate}
	In both cases, $\textbf{u} \in \Z^{\beta_1}$ denotes the image of the abelianization of $ u \in H $ under the projection $p$, where $p: H^{ab} \to \Z^{\beta_1}$ is the natural projection from $H^{ab}$ to its free part $\Z^{\beta_1}$.
\end{prop}
\begin{proof}
Let $\{x_i\}_{i=1}^n$ and $\{t_i\}_{i=1}^m$ denote the generating sets of $H$ and $\Z^m$, respectively. We consider the images of generators under the endomorphism $\Psi$. Now, we set
	\begin{eqnarray*}
        \Psi(t_i)=z_i\textbf{t}^{\textbf{q}_{i}},  \quad 1\leq i\leq m.
	\end{eqnarray*}
Let $p_1:H\times\mathbb{Z}^m\to H, ~u\textbf{t}^{\textbf{a}}\mapsto u,$ ~$p_2:H\times\mathbb{Z}^m\to \Z^m, ~u\textbf{t}^{\textbf{a}}\mapsto \textbf{t}^{\textbf{a}}$ and $\phi_i:=(p_i\comp\Psi)|_{H}$ for $i=1,2$. Then $\phi_1\in \edo(H)$, and $\phi_2$ can be expressed as
        $$\phi_2: H \xrightarrow{\lambda} H^{ab} \xrightarrow{p} \Z^{\beta_1}\xrightarrow{P} \Z^m$$
by the universal property of abelianization, where $p$ is the natural projection and $P$ is a homomorphism from $\Z^{\beta_1}$ to $\Z^m$, represented by a matrix $\textbf{P} \in \mathcal{M}_{\beta_1\times m}(\Z)$. Thus
$$\phi_2(u)=P \comp p \comp \lambda(u)=\textbf t^\textbf{uP},$$
 where $\textbf{u}=p\comp \lambda(u) \in \Z^{\beta_1}$.
There are two cases:

(1) All the $z_i$'s are trivial, namely
        $$\Psi(t_i)=\textbf{t}^{\textbf{q}_{i}}, \quad 1\leq i\leq m.$$
Then the homomorphism $\Psi$ can be expressed as
	$$ \Psi=\Psi_{\phi_1,\textbf{Q},\textbf{P}}: u\textbf{t}^\textbf{a} \mapsto \phi_1(u)\textbf{t}^{\textbf{aQ}+\textbf{uP}}, \quad\forall u\textbf{t}^\textbf{a}\in H\times\mathbb{Z}^m,$$
	where $\textbf{Q}=(\textbf{q}_{1},\dots,\textbf{q}_{m})^T\in \mathcal{M}_m(\mathbb{Z})$, $\textbf{P} \in \mathcal{M}_{\beta_1 \times m}(\mathbb{Z})$ and $\textbf{u}=p\comp \lambda(u) \in \Z^{\beta_1}$.
	
(2) Some $z_i\neq 1$. We assume $\phi_1(x_k)=c_k, ~1\leq k\leq n.$
Since $t_i$ commutes with all the $t_j$'s and $x_k$'s, we have
	$$z_iz_j=z_jz_i, \quad z_ic_k=c_kz_i.$$
By Lemma \ref{lem coummte same axis}, there exists an element (not a proper power) $z\in H$ such that
$$1\neq \left\langle z_1,\dots, z_m, c_1, \ldots, c_{n}\right\rangle \subset \left\langle z \right\rangle \cong \Z.$$
Let $\textbf{l}=(l_1, \ldots, l_m)\in \mathbb{Z}^m$ such that $z_i=z^{l_i}$.
Then $\phi_1$ can be expressed as
          $$\phi_1: H \xrightarrow{\lambda} H^{ab} \xrightarrow{p} \Z^{\beta_1}\xrightarrow{h} \langle z \rangle$$
by the universal property of abelianization, where $h$ is a homomorphism from $\Z^{\beta_1}$ to $\Z$, represented by a matrix $\textbf{h} \in \Z^{\beta_1}$. Thus $\phi_1(u)=h\comp p\comp \lambda(u)=z^{\textbf{uh}^T}$, and we get
	$$\Psi=\Psi_{z,\textbf{l},\textbf{h},\textbf{Q},\textbf{P}}: u\textbf{t}^\textbf{a}\mapsto z^{\textbf{al}^T+\textbf{uh}^T}\textbf{t}^{\textbf{aQ}+\textbf{uP}},$$
	where $1\ne z\in H$, $\textbf{Q}\in \mathcal{M}_m(\mathbb{Z})$, $ \textbf{P}\in \mathcal{M}_{\beta_1\times m}(\mathbb{Z})$, $\textbf{0}\ne \textbf{l} \in \mathbb{Z}^m$, and $\textbf{h}\in \mathbb{Z}^{\beta_1}$.
\end{proof}

\subsection{Fixed subgroups}
Now, we can easily obtain the following corollary on fixed subgroups.

\begin{cor}\label{cor betti=0}
    Let $H$ be a non-elementary torsion-free hyperbolic group. If the first betti number $\beta_1(H)=0$, then every end-fixed subgroup of $G=H\times\Z^k\ (k\geq 1)$ is isomorphic to $\Z^s\ (s\leq k+1)$, or $N\times \Z^s\ (s\leq k)$ for $N\subset H$ an end-fixed subgroup of $H$.
\end{cor}
\begin{proof}
Let $\Psi$ be an endomorphism of $G=H\times\Z^k$. Then, according to Proposition \ref{des of hyperbolic-end}, $\Psi=\Psi_{\phi,\textbf{Q},\textbf{P}}$ or $\Psi_{z,\textbf{l},\textbf{h},\textbf{Q},\textbf{P}}$. Since $\beta_1(H)$ is zero, we have $\textbf{P}=\textbf{0}$ and $\textbf{uh}^T=0$, i.e.,
$$\Psi: u\textbf{t}^\textbf{a}\mapsto \phi(u)\textbf{t}^\textbf{aQ}\ \mbox{ or }\ u\textbf{t}^\textbf{a}\mapsto z^{\textbf{al}^T}\textbf{t}^\textbf{aQ}.$$
Now we have two cases:

Case ($\Psi=\Psi_{\phi,\textbf{Q},\textbf{0}}$). Then
\begin{eqnarray*}
    \fix\Psi&=&\{u\textbf{t}^\textbf{a}\mid u\textbf{t}^\textbf{a}=\phi(u)\textbf{t}^\textbf{aQ}\}\\
    &=&\{u\textbf{t}^\textbf{a}\mid u=\phi(u), ~\textbf{a}(\textbf{I}-\textbf{Q})=\textbf{0} \}\\
    &\cong& \fix\phi\times\Z^s,
\end{eqnarray*}
where $s\leq k$.

Case ($\Psi=\Psi_{z,\textbf{l},\textbf{0},\textbf{Q},\textbf{0}}$). Then
$$\fix\Psi\subset\langle z\rangle\times\Z^k\cong \Z^{k+1},$$
and hence $\fix\Psi\cong\Z^s$ for some $s\leq k+1$.
\end{proof}

Moreover, when $\beta_1(H)$ is non-zero, the following proposition \cite[Proposition 6.1 (2)]{LWZ25} shows the existence of infinitely many aut-fixed subgroups of $H\times\Z^k\ (k\geq 2)$.

\begin{prop}[Lei-Wang-Zhang, \cite{LWZ25}]\label{lwz25 prop 6.1 (2)}
    Let $H$ be a non-elementary torsion-free hyperbolic group and $G=H\times\Z^k\ (k\geq 2)$. If $\beta_1(H)$ is non-zero, then for any $m\geq1$ and any $0\leq s<k$, there exists $\phi\in\aut(G)$ such that $\fix \phi\cong N\times \Z^s$, where $N$ is a normal subgroup of $H$ with index $[H:N]=m$.
\end{prop}

Besides, for $H\times \Z$, we have a similar results about the end-fixed subgroups by a similar construction as in \cite[Proposition 6.1 (2)]{LWZ25}.

\begin{prop}\label{prop end-fixed subgps of H times Z}
    Let $H$ be a non-elementary torsion-free hyperbolic group and $G=H\times \Z$. If $\beta_1(H)$ is non-zero, then for any $m\geq 1$, there exists $\phi\in\edo(G)$ such that $\fix \phi \cong N$, where $N$ is a normal subgroup of $H$ with index $[H:N]=m$.
\end{prop}
\begin{proof}
Denote the elements of $H\times\Z$ as the form of $(u,t)$ for every $u\in H$ and $t\in\Z$.
    Since $\beta_1(H)$ is non-zero, then there exists an epimorphism $\gamma:H\to \Z$.
    For any $m\geq 1$, we can define an endomorphism:
    $$\phi_m:H\times \Z\to H\times\Z, ~~(u,t)\mapsto (u,\gamma(u)+(m+1)t).$$
    Then,
    \begin{eqnarray*}
        \fix(\phi_m)&=&\{(u,t)\in H\times\Z\mid t=\gamma(u)+(m+1)t\}\\
        &=&\{(u,t)\in H\times \Z\mid \gamma(u)=-mt\}\\
        &\cong&\{u\in H\mid \gamma(u)\equiv 0\mod m\}:=N,
    \end{eqnarray*}
where $N$ is a normal subgroup of $H$ with index $[H:N]=m\geq 1$.
\end{proof}

A group is called \textit{finite index rigid} if it does not contain isomorphic finite index subgroups of different indices. Lazarovich \cite{La25} showed:

\begin{thm}[Lazarovich]\label{finite index rigid}
    Every non-elementary hyperbolic group is finite index rigid.
\end{thm}

By Theorem \ref{finite index rigid}, the normal subgroups $N$'s constructed in Proposition \ref{lwz25 prop 6.1 (2)} and Proposition \ref{prop end-fixed subgps of H times Z}, are not isomorphic for different $m$. Thus, all the constructed fixed subgroups of $H\times \Z^k (k\geq 1)$ are not isomorphic. Moreover, we can obtain the proof of Theorem \ref{thm hyper. times Z^m finite end-fixed subgps} as follows.

\begin{proof}[\textbf{Proof of Theorem \ref{thm hyper. times Z^m finite end-fixed subgps}}]
$``\Rightarrow"$: If $\beta_1(H)$ is non-zero, Proposition \ref{lwz25 prop 6.1 (2)} and Proposition \ref{prop end-fixed subgps of H times Z} show the existence of, up to isomorphism, infinitely many end-fixed subgroups of $G$, so $\beta_1(H)=0$. Once $H$ contains, up to isomorphism, infinitely many end-fixed subgroups, we can construct the corresponding, up to isomorphism, infinitely many end-fixed subgroups of $G$ by taking $$\edo(G)\ni \Psi=\Psi_{\phi,\textbf{0}, \textbf{0}}:u\textbf{t}^\textbf{a}\mapsto \phi(u)\textbf{t}^\textbf{0}, ~\mbox{for}~ \phi\in\edo(H).$$ Therefore, $H$ contains, up to isomorphism, finitely many end-fixed subgroups.

 $``\Leftarrow"$: Denote the end-fixed subgroups of $H$ as $N_i\ (i=1,\dots,n)$, up to isomorphism. Since $\beta_1(H)=0$, then, by applying Corollary \ref{cor betti=0}, we get that every end-fixed subgroup of $G$ is always isomorphic to one of the groups in the following finite set
 $$\{N_i\times\Z^s\mid i=1,\dots,n, ~0\leq s\leq k\}\cup\{\Z^s\mid 0\leq s\leq k+1\}.$$
 Therefore, $G$ contains, up to isomorphism, finitely many end-fixed subgroups.
 \end{proof}

\textbf{Acknowledgements.} The authors thank Peng Wang for providing useful communication and help in the process of writing.


\end{document}